\documentclass[a4paper]{article}
\usepackage{lrmacros_en}
\title{A note on the geometry of the MAP partition in Conjugate Exponential Bayesian Mixture Models}
\author{Łukasz Rajkowski\footnote{Faculty of Mathematics, Informatics and
Mechanics, University of Warsaw}\\ 
John Noble\footnote{Faculty of Mathematics, Informatics and Mechanics, University of Warsaw}}

\usepackage{multirow}
\usepackage{array}
\usepackage{longtable}
\usepackage{natbib}
\usepackage{caption}
\usepackage{subcaption}
\usepackage{float}
\usepackage{upgreek}

\newcolumntype{S}{>{\centering\arraybackslash}m{4cm}}

\renewcommand{\bT}{\bm{T}}
\newcommand{\bta}{\bm{\eta}}
\newcommand{\bma}{\bm{a}}
\newcommand{\bt}{\bm{t}}
\newcommand{\bft}{\mathbf{t}}
\newcommand{\btheta}{\bm{\theta}}
\newcommand{\bftheta}{\bm{\uptheta}}
\newcommand{\bphi}{\bm{\phi}}
\newcommand{\bfphi}{\bm{\upphi}}
\newcommand{\btau}{\bm{\tau}}
\newcommand{\bchi}{\bm{\chi}}
\newcommand{\bmu}{\bm{\mu}}
\newcommand{\bLambda}{\bm{\Lambda}}
\newcommand{\bSigma}{\bm{\Sigma}}
\newcommand{\bPsi}{\bm{\Psi}}

\renewcommand{\diag}{\textrm{diag}}
\newcommand{\low}{\textrm{low}}
\newcommand{\bmB}{\bm{B}}
\newcommand{\Wishart}{\cW}
\newcommand{\bx}{\bm{x}}
\newcommand{\bfx}{\mathbf{x}}
\newcommand{\by}{\bm{y}}

\newcommand{\ERP}{\textrm{ERP}}

\newcommand{\LL}{\cU}
\newcommand{\ov}[1]{\overline{#1}}
\newcommand{\un}[1]{\underline{#1}}
\newcommand{\bz}{\bm{z}}
\newcommand{\mnt}[1]{{ \hat{#1} }}

\newcommand{\Ik}{{\color{black}k}}

\newcommand{\convset}{V}

\newcommand{\uvar}{\bm{u}}
\newcommand{\buvar}{\mathbf{u}}

\newcommand{\trs}{\hspace*{-.5mm}^\top}

\makeatletter
\newcommand*{\shifttext}[2]{%
  \settowidth{\@tempdima}{#2}%
  \makebox[\@tempdima]{\hspace*{#1}#2}%
}
\makeatother

\numberwithin{equation}{section}

\theoremstyle{plain}
\newtheorem{thm}{Theorem}
\newtheorem*{thm*}{Theorem}
\numberwithin{thm}{section}

\newtheorem{lem}[thm]{Lemma}

\newtheorem{cor}[thm]{Corollary}

\newtheorem{rem}[thm]{Remark}

\theoremstyle{definition}
\newtheorem{dfn}[thm]{Definition}
\newtheorem*{dfn*}{Definition}
\newtheorem*{not*}{Notation}
\newtheorem*{rem*}{Remark}
\newtheorem{exmp}[thm]{Example}

\begin{document}
\maketitle
\abstract{We investigate the geometry of the maximal a posteriori (MAP) partition in
the Bayesian Mixture Model where the component and the base distributions are
chosen from conjugate exponential families. We prove that in this case the
clusters are separated by the contour lines of a linear functional of the
sufficient statistic. As a particular example, we describe Bayesian Mixture of
Normals with Normal-inverse-Wishart prior on the component mean and covariance,
in which the clusters in any MAP partition are 
separated by a quadratic surface. In connection with results of
\cite{bib:Rajkowski2018}, where the linear separability of clusters in the
Bayesian Mixture Model with a fixed component covariance matrix was proved, it gives a nice Bayesian analogue of the
geometric properties of Fisher Discriminant Analysis (LDA and QDA). 
}


\setlength\parindent{0pt}
\section{Introduction}\label{sec:intro}
\noindent In the standard setting of Bayesian Mixture Models we assume that the
target distribution is a random mixture of distributions from some parametrized
family. We
assume that the probabilities of components are sampled from a (perhaps
infinitely dimensional) simplex and the parameters of the component distribution are
sampled independently for each component.
A prominent example is the Dirichlet Process Mixture Model
(\cite{bib:Antoniak1974mixtures}), where the prior distribution on the
probability weights is Sethuraman's stick breaking process (\cite{bib:Sethuraman1994}).

\smallskip
A popular choice of the component distribution is multivariate Normal, which
gives \textit{Normal Bayesian Mixture Model}. There are two
standard conjugate prior distributions on the component mean and
covariance matrix (\cite{bib:Gelman2013bayesian}, Chapter 3.6): Normal distribution on the mean with fixed component
covariance matrix or Normal-inverse-Wishart distribution (here the terminology
is adopted from \cite{bib:Murphy2007conjugate}), where the component
covariance matrix follows the inverse-Wishart distribution and the component
mean (conditioned on the component covariance matrix) is Normal.
The exact specification of these priors is given in \Cref{sec:three_models}.


\smallskip
Bayesian Mixture Models give a basis for cluster analysis. Indeed, one can
translate the distribution on component probabilities into a discrete prior
distribution on the possible data partitions -- just like Sethuraman's
construction translates in to the Generalised P\'{o}lya Urn Scheme
(\cite{bib:Blackwell1973ferguson}), also known as the Chinese Restaurant Process
(\cite{bib:Aldous1985exchangeability}).
The inference about clusters is based on the posterior distribution on the space
of partitions. Analysing partition which maximises the posterior probability
(\textit{the MAP partition}) seems to be a natural choice.

\smallskip
In \cite{bib:Rajkowski2018} it is proved that in the Normal Bayesian Mixture
Model, when the component covariance matrix is fixed and the prior on the component
mean is Normal, the MAP partition is \textit{convex}, i.e. the convex hulls of
clusters are disjoint. An equivalent formulation is: for every two clusters in the MAP there
exists a hyperplane that separates them.

\smallskip
Placing an inverse-Wishart prior on the cluster covariance structure, with covariances
for different clusters, independent of each other, gives better modelling
possibilities, since it is unusual for the covariance to be known a priori and
the same for different clusters.
It would be of interest to characterise the boundaries of the MAP partition in
this case. Since cluster covariance structures are no longer fixed, we might expect quadratic
boundaries, analogously to the Fisher's Quadratic Discriminant Analysis
(\cite{bib:Hastie2001elements}). 

\smallskip
A natural generalisation of this hypothesis is to establish the separability result
in the case when the base and the component distributions in the Bayesian
Mixture Model form a conjugate exponential family. One may expect that the
separability can be described in terms of the sufficient statistics.
This is indeed what happens and it is the goal of the article to prove this.

\section{Formal statement of the result}\label{sec:result}

\textbf{Naming conventions and notation.}
In order to facilitate the readership, we introduce the following naming
convention. To make a distinction between real numbers and vectors or matrices, we denote
the latter in bold, i.e. $\bx, \btheta, \bm{\bSigma}$. We do the same to
distinguish between real-valued, and vector-valued functions. We use an upright
bold font do denote a sequence of
vectors, e.g. $\mathbf{x}=(\bx_1,\ldots,\bx_n)$. In such case, when
$I$ is a subset of $[n]:=\{1,\ldots,n\}$, we define
$\bfx_I:=(\bx_i)_{i\in I}$.

\subsection{Bayesian Mixture Models and the MAP partition}
Let $\Theta\subset\R^p$ be the parameter space and $\{ g_{ \btheta }\colon
\btheta\in\Theta \}$ be
a family of probability densities on the observation space $\R^d$. Consider a
prior distribution on $\Theta$ given by its density $\pi$. 
Let $\cP$ be a probability
distribution on the $m$-dimensional simplex
$\Delta^m=\{\bm{p}=(p_i)_{i=1}^m\colon \textrm{$\sum_{i=1}^m p_i=1$ and
$p_i\geq 0$ for $i\leq m$}\}$ (where $m\in\N\cup\{\infty\}$). Let 
\begin{equation}\label{eq:bmm1}
\begin{array}{rcl}
\bm{p}=(p_i)_{i=1}^m&\sim& \cP \\
\bftheta=(\btheta_i)_{i=1}^m&\iid& \pi \\
\bfx=(\bx_1,\ldots,\bx_n) \cond \bm{p},\bftheta&\iid& \sum_{i=1}^m p_i
g_{\btheta_i}.
\end{array}
\end{equation}
This is a \textit{Bayesian Mixture Model}.
It can model possible clusters within data; they are defined by deciding which
$g_{\btheta_i}$ generated a given data point. In order to formally define the
clusters, we need to rewrite \eqref{eq:bmm1} as
\begin{equation}\label{eq:bmm2}
\begin{array}{rcl}
\bm{p}=(p_i)_{i=1}^m&\sim& \cP \\
\bftheta=(\btheta_i)_{i=1}^m&\iid& \pi \\
\bfphi=(\bphi_1,\ldots,\bphi_n) \cond \bm{p},\bftheta&\iid& \sum_{i=1}^m p_i
\delta_{\btheta_i}\\
\bx_i\cond \bm{p},\bftheta,\bfphi&\sim& g_{\bphi_i} \quad\textrm{ independently
for all $i\leq n$.}
\end{array}
\end{equation}
Then the clusters are the
classes of abstraction of the equivalence relation $i\sim j\equiv \bphi_i=\bphi_j$.
In this way the distribution on $m$ dimensional simplex \textit{generates} a
probability distribution on the partitions of set $[n]$ into at most $m$ subsets. According to
\cite{bib:Pitman2002combinatorial} this leads to an
\textit{exchangeable partition}, i.e. a random partition whose probability
function is invariant with respect to permutations of indices.

\begin{dfn}
We say that $\bm{\Pi}$ is an exchangeable random partition of $[n]$ if for every
partition $\cI$ of $[n]$ and permutation $\sigma\colon [n]\to [n]$,
\begin{equation}\label{eq:expart}
\P(\bm{\Pi}=\cI) = 
\P\Big(\bm{\Pi}=\big\{\{\sigma(i)\colon i\in I\}\colon I\in\cI\big\}\Big).
\end{equation}
In order to indicate that $\bm{\Pi}$ is an exchangeable random partition of $[n]$ we use
a generic notation $\bm{\Pi}\sim \ERP_n$. Moreover we use the notation
$p_n(\cI):=\P(\bm{\Pi}=\cI)$.
\end{dfn}

For $\btheta\sim \pi$, $k\in\N$ and $\buvar=(\uvar_1,\ldots,\uvar_k)\cond \btheta\iid g_{ \btheta }$ let
$f_k$ be the resulting marginal distribution on $\buvar$, i.e. 
\begin{equation}
f_{ k }(\uvar_1,\ldots,\uvar_k):=\int_\Theta \pi(\btheta)\prod_{i=1}^k g_{ \btheta }(\uvar_i)\d{\btheta}.
\end{equation}
Let $\ERP_n$ be the exchangeable probability distribution on the space of
partitions generated by $\cP$. We see that \eqref{eq:bmm1} is equivalent to
\begin{equation}\label{eq:bmm3}
\begin{array}{rcl}
\cI&\sim& \ERP_n\\
\bfx_I:=(\bx_i)_{i\in I}\cond \cI &\sim& f_{ |I| } \quad\textrm{ independently for
all $I\in\cI$.}
\end{array}
\end{equation}

We stress the fact that the independent sampling on the `lower' level of
\eqref{eq:bmm3} relates
to the independence between clusters (conditioned on the random partition); within one cluster the
observations are (marginally) dependent. To make the notation more concise we
define
\begin{equation}\label{eq:pooled_f}
f(\bfx\cond \cI):= \prod_{I\in\cI} f_{|I|}(\bfx_I).
\end{equation}
Then
\eqref{eq:bmm3} becomes
\begin{equation}\label{eq:bmm3a}
\begin{array}{rcl}
\cI&\sim& \ERP_n\\
\bfx\cond \cI &\sim& f(\cdot\cond \cI).
\end{array}
\end{equation}

\begin{exmp}
Consider the Dirichlet Process Mixture Model (\cite{bib:Antoniak1974mixtures}).
Let $\alpha>0$, $G_0$ be a probability measure on $\Theta$ with density
$\pi$ and $DP(\alpha,G_0)$ be the Dirichlet Process on $\Theta$
(\cite{bib:Ferguson1973bayesian}). The Dirichlet
Process Mixture Model is defined by
\begin{equation}\label{eq:dpmm1}
\begin{array}{rcl}
G&\sim &DP(\alpha, G_0)\\
\bfphi=(\bphi_1,\ldots,\bphi_n) \cond G&\iid& G\\
\bx_i\cond G,\bfphi&\sim& g_{\bphi_i} \quad\textrm{ independently
for all $i\leq n$.}
\end{array}
\end{equation}
Let $V_1,V_2,\ldots\iid \Beta(1,\alpha)$, $p_1=V_1$,
$p_k=V_k\prod_{i=1}^{k-1}(1-V_i)$ for $k>1$. By
\cite{bib:Sethuraman1994} by setting $\cP$ to be the distribution of
$\bm{p}=(p_1,p_2,\ldots)$ we get that \eqref{eq:dpmm1} is
equivalent to \eqref{eq:bmm1}. The exchangeable random partition that $\cP$
generates is the Generalized Polya Urn Scheme (\cite{bib:Blackwell1973ferguson})
or the Chinese Restaurant Process (\cite{bib:Aldous1985exchangeability}) with
the probability weight given by
\begin{equation}
p_n(\cI)=\frac{\alpha^{|\cI|}}{\alpha^{(n)}}\prod_{I\in\cI}(|I|-1)!,
\end{equation}
where $\alpha^{(n)}=\alpha(\alpha+1)\ldots
(\alpha+n-1)$. Again, setting $\ERP_n$ to be the Chinese Restaurant Process with
parameter $\alpha$ we get the equivalence between \eqref{eq:dpmm1} and
\eqref{eq:bmm3a}.
\end{exmp}

\begin{dfn}
Let $\bfx=(\bx_1,\ldots,\bx_n)\in(\R^d)^n$. We say that a partition  $\hat{\cI}$ of $[n]$ is a
MAP of $\bfx$ if for any other partition $\cI$ of $[n]$:
\begin{equation}\label{eq:MAPdef}
 p_n(\cI)f(\bfx_I\cond \cI)\leq
 p_n(\hat{\cI}) f(\bfx_I\cond \hat{\cI}).
\end{equation}
\end{dfn}

\begin{not*}
Here and below, $\argmax_{a\in A}\bphi(a)$ is the set
of maximisers of function $\bphi$ on the set $A$ (note that the maximiser may not
be unique). Hence the MAP partition of $\bfx$ in a Bayesian Mixture Model
\eqref{eq:bmm3a} can be defined by
\begin{equation}\label{eq:MAPdef_argmax}
\hat{\cI}\in\argmax_{\textrm{partitions $\cI$ of $[n]$}} p_n(\cI) f(\bfx_I\cond
\cI).
\end{equation}
or, equivalently, using \eqref{eq:pooled_f}
\begin{equation}\label{eq:MAPdef_log_argmax}
\hat{\cI}\in\argmax_{\textrm{partitions $\cI$ of $[n]$}} \Big(\ln p_n(\cI) +
\sum_{I\in\cI} \ln f_{|I|}(\bfx_I)\Big).
\end{equation}
\end{not*}
\subsection{Specification of the Exponential Family Bayesian Mixture
Models}\label{sec:exp_models}
In the paper we consider \textit{Exponential Family} Bayesian Mixture Models in
which the base and the component distributions come from a conjugate exponential
family.

\smallskip
Let $\cX\subseteq \R^d$ be the observation space and let $\Theta\subseteq \R^p$
be the parameter space, which is an open subset of $\R^p$. Let $\bT\colon
\cX\to\R^p$ be a statistic, let $h\colon \cX\to\R$ be a function with the
support of positive Lebesgue measure and let $\bta\colon \Theta\to\R^p$.
Let $B\colon \Theta\to\R$.
Suppose that the family of component densities is given by
\begin{equation}\label{eq:compdens}
g_{ \btheta }(\bx)=h(\bx)\cdot\exp\left\{\bT(\bx)\trs\bta(\btheta)-B(\btheta)\right\}.
\end{equation}
Let $B(\btheta)=\bma\trs\bmB(\btheta)$ where $\bma\in\R^q$ and
$\bmB(\btheta)=[B_1(\btheta),\ldots,B_q(\btheta)]\trs$.
Let $H\colon \Theta\to\R$ and let
\begin{equation}
\Omega:=\left\{(\bchi,\btau)\in \R^p\times\R^q\colon
\int_{ \Theta } H(\btheta)\cdot\exp\left\{[
\bta(\btheta)\trs, \bmB(\btheta)\trs ]
\begin{bmatrix} \bchi\\ \btau \end{bmatrix}\right\}\d{\btheta}<\infty\right\}.
\end{equation}
be a nonempty hyperparameter space. 
We define $A\colon \R^p\times\R^q\to \R$ as 
\begin{equation}
A( \bchi,\btau ):=
\ln\left(\int_{ \Theta } H(\btheta)\cdot\exp\left\{[
\bta(\btheta)\trs, \bmB(\btheta)\trs ]
\begin{bmatrix} \bchi\\ \btau \end{bmatrix}\right\}\d{\btheta}\right)
\end{equation}

\begin{dfn*}
\textit{Canonical Exponential Family Bayesian Mixture Model} is a Bayesian Mixture Model
in which the component density is given by \eqref{eq:compdens} and the base
density is 
\begin{equation}\label{eq:expdef}
\pi(\btheta)=\pi_{\bchi,\btau}(\btheta)=H(\btheta)\cdot\exp\left\{[
\bta(\btheta)\trs, -\bmB(\btheta)\trs ]
\begin{bmatrix} \bchi\\ \btau \end{bmatrix}
-A( \bchi,\btau  )\right\}
\end{equation}
for some $(\bchi,\btau)\in\Omega$.
\end{dfn*}

Let $\btheta\sim\pi_{\bchi,\kappa}$ and $\buvar=(\uvar_1,\ldots,\uvar_k)\cond
\btheta\sim g_{ \btheta }$ then it follows from the multiplication rule that
$(\bchi_\buvar,\btau_k)\in\Omega$ almost surely and $\btheta\cond \buvar\sim \pi_{\bchi_\buvar,\btau_k}$,
where $\bchi_\buvar:=\bchi+\sum_{i=1}^k \bT(\uvar_i)$ and
$\btau_k:=\btau+k\bma$.
As the marginal density of $\buvar$ is the quotient of the joint density
of $(\btheta,\buvar)$ and the conditional density of $\btheta\cond\buvar$, we get that
\begin{equation}\label{eq:expmarg}
\buvar\sim f_k(\buvar):= \prod_{i=1}^k h(\uvar_i)\cdot
\exp\left\{A( \bchi_\buvar, \btau_k )-A( \bchi, \btau )\right\}
\end{equation}

%
%

\subsection{Main result}
We now define what we mean by $\bT$-linear separation of clusters.

\begin{dfn}
Let $\cX$ be a family of subsets of $\R^d$ and $\cL$ a family of real functions on
$\R^d$. We say that \emph{$\cX$ is separated by $\cL$} if for every $X,Y\in\cX$,
$X\neq Y$, there exist 
$L_{X,Y}\in\cL$ such that $L_{X,Y}(\bx)\geq 0$ and $L_{X,Y}(\by)< 0$ for all $\bx\in X, \by\in Y$.
Moreover, if $\cL=\{\bm{a}\trs \bT(\bx) + b\colon \bm{a}\in \R^s, b\in\R\}$ for
some function $\bT\colon \R^d\to\R^s$, we say that \textit{$\cX$ is $\bT$-linearly
separated from $\cY$}. If $\bT(\bx)=\bx$, we use the term \textit{linear
separability} for short.
\end{dfn}

\begin{figure}[H]
\begin{center}
\begin{subfigure}{.25\textwidth}
\centering
\includegraphics[width=\textwidth]{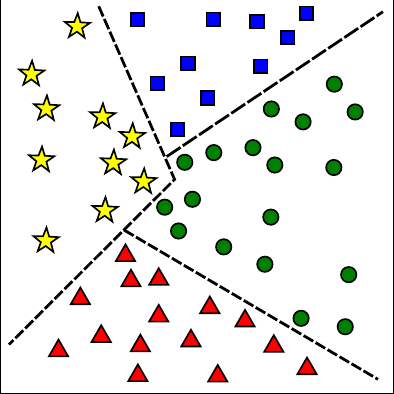}
\caption{This family is linearly separable.}
\end{subfigure}\qquad
\begin{subfigure}{.25\textwidth}
\centering
\includegraphics[width=\textwidth]{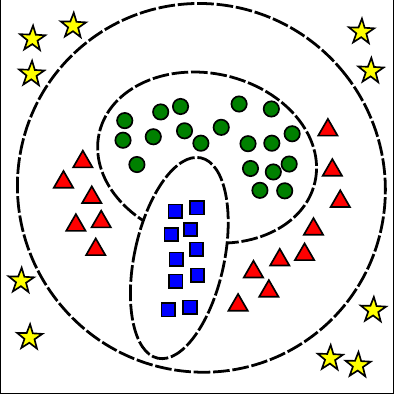}
\caption{This family is quadratically, but not linearly separable.}
\end{subfigure}\qquad
\begin{subfigure}{.25\textwidth}
\centering
\includegraphics[width=\textwidth]{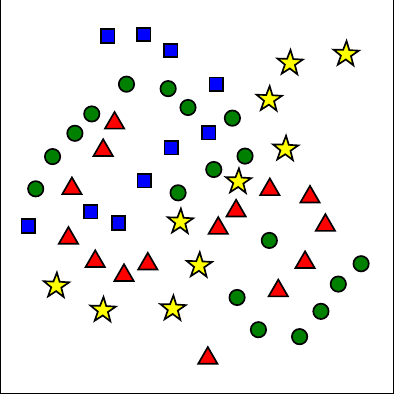}
\caption{This family is not quadratically separable.}
\end{subfigure}
\caption{
Illustration of the different types of separability.
}
\end{center}
\label{fig:convex}
\end{figure}
\begin{not*}
For the notational convenience we will use the separability notions also with
respect to the sets of sequences in $\R^d$. For example, if $\bx_1,\ldots,\bx_n\in
\R^d$ and $I,J$ are disjoint subsets of $[n]$ then the expression
\emph{$\bfx_I$ is linearly separated from $\bfx_J$} means that
$\{\bx_i\colon i\in I\}$ is linearly separated from $\{\bx_j\colon j\in J\}$.
\end{not*}

In \cite{bib:Rajkowski2018} it is proved that in the
Normal Bayesian Mixture Model with Normal distribution on the component mean and
fixed covariance matrix, when the prior on the space of partitions is the
Chinese Restaurant Process, the convex hulls of the clusters in the MAP
partition are disjoint. Equivalently, the MAP is linearly separable. The
following theorem is a generalisation of this result to the case of arbitrary
Exponential Family Bayesian Mixture Model.

\begin{thm}\label{res:Tlinsep}
Let $\bx_1,\ldots,\bx_n\in\R^d$ be pairwise distinct and let $\hat{\cI}$ be the MAP
partition of $\bx_1,\ldots,\bx_n$ in the Bayesian Mixture Model where the
prior on component parameters is given by \eqref{eq:expdef}.
Then the family $\{\bfx_I\colon I\in\hat{\cI}\}$ is $\bT$-linearly separable.
\end{thm}
\begin{proof}
The proof is left for \Cref{sec:proofs}.
\end{proof}

\subsection{Example: Normal Bayesian Mixture Models}\label{sec:three_models}
As a commonly used in practice example, we consider \textit{Normal} Bayesian Mixture Models in which the
component distributions are multivariate Normal, so $\cX=\R^d$. The mean is
uknown, but the covariance matrix may be treated as uknown, known or known up to
a scaling factor.

\begin{not*}
We use two standard notations to denote the determinant of a square matrix
$\bLambda$: $\det \bLambda$ and $|\bLambda|$. The latter may seem ambiguous as we
also use the symbol $|\cdot|$ to denote the cardinality of a set. However, the
meaning of this symbol is always clear from the context.
\end{not*}

\begin{not*}
To keep the notation precise, in the following we introduce the
following convention: if $\bSigma$ is a symmetric $d\times d$ matrix, 
then $\diag(\bSigma)$ is the diagonal of $\bSigma$, treated as $d$-dimensional
vector, and $\low(\bSigma)$ is the `lower triangular' part of $\bSigma$, treated
as a $\frac{d(d-1)}{2}$ dimensional vector, whose
$\left(\frac{(i-1)(i-2)}{2}+j\right)$-th coordinate is equal to $(i,j)$-th
coefficient of $\bSigma$, where $i>j$.
\end{not*}
\subsubsection{Normal-inverse-Wishart}
In this case both the mean and the covariance matrix are unknown. The parameter
space is therefore equal to $\Theta=\R^d\times \cS^d_+$, where $\cS^d_+$ is the
space of all positive definite, $d\times d$ matrices.
This can be naturally interpreted as an open subset of $\R^p$, where
$p=\frac{d(d-1)}{2}+d$.
For $\btheta=(\bmu,\bLambda)\in\Theta$ the component distribution is
$\bx\cond\btheta\sim\Normal(\bmu,\bLambda)$ and

\begin{equation}\label{eq:NIWmodel}
\begin{array}{rcl}
\bLambda&\sim&\Wishart^{-1}(\nu_0+d+1,\nu_0\bSigma_0)\\
\bmu\cond\bLambda&\sim&\Normal(\bmu_0,\bLambda/\kappa_0),
\end{array}
\end{equation}
where $\Wishart^{-1}$ is the inverse-Wishart distribution. Here the hyperparameters are $\kappa_0,\nu_0>0$, $\bmu_0\in\R^d$ and
$\bSigma_0\in\cS^+$. This model is listed in \cite{bib:Gelman2013bayesian} with a slightly different
hyperparameters, but we made this modification to obtain
\begin{equation}\label{eq:NIW_expect}
\begin{array}{rl}
\E(\bV(\bx\cond \bmu,\bLambda))&=\E\bLambda=\bSigma_0,\\
\bV(\E(\bx\cond \bmu,\bLambda))&=\bV(\bmu)=\E\bV(\bmu\cond \bLambda)+\bV\E(\bmu\cond\bLambda)=\E
\bLambda/\kappa_0+\bV(\bmu_0)=\bSigma_0/\kappa_0,
\end{array}
\end{equation}
which is consistent with the remaining two priors.

The conditional densities are given by
\begin{equation}\label{eq:NIWmodel_dens}
\begin{array}{rcl}
\bx\cond\bmu,\bLambda&\sim&
(2\pi)^{-d/2}|\bLambda|^{-1/2}
\exp\left\{-\re{2}(\bx-\bmu)\trs\bLambda^{-1}(\bx-\bmu)\right\}\\
\bmu\cond\bLambda&\sim&
(2\pi)^{-d/2}\kappa_0^{d/2}|\bLambda|^{-1/2}
\exp\left\{-\re{2}\kappa_0(\bmu-\bmu_0)\trs\bLambda^{-1}(\bmu-\bmu_0)\right\}\\
\bLambda&\sim&
\left(\frac{|\nu_0\bSigma_2|}{2^d}\right)^{\frac{\nu_0+d+1}{2}}
\Gamma_d\left(\frac{\nu_0+d+1}{2}\right)^{-1}
|\bLambda|^{-\frac{\nu_0+2d+2}{2}}
\exp\left\{-\re{2}\nu_0\tr(\bSigma_0\bLambda^{-1})\right\}
\end{array}
\end{equation}

The density of $\bx\cond\btheta$ can be expressed as \eqref{eq:compdens} by
placing
\begin{equation}\label{eq:alladyn}
\begin{split}
h(\bx)= ( 2\pi )^{-d/2}, \quad
\bT(\bx)&=\begin{bmatrix}-\re{2}\diag( \bx\bx\trs )\\ -\,\low( \bx\bx\trs )\\\bx\end{bmatrix}, \quad
\bta(\btheta) = \begin{bmatrix}\diag(\bLambda^{-1}) \\ 
\low(\bLambda^{-1}) \\
\bLambda^{-1}\bmu\end{bmatrix},\\
B(\btheta)&=\re{2}\ln|\bLambda|+\re{2}\bmu^t\bLambda^{-1}\bmu
\end{split}
\end{equation}

We get $\eqref{eq:expdef}$ by placing $\bma=[1,1]\trs$,
\begin{equation}
\bmB(\btheta)=\begin{bmatrix}\re{2}\ln|\bLambda|\\\re{2}\bmu^t\bLambda^{-1}\bmu\end{bmatrix},\quad
\btau= \begin{bmatrix} \nu_0+2d+3\\ \kappa_0 \end{bmatrix},\quad
\bchi=\begin{bmatrix}-\re{2}\diag( \nu_0\bSigma_0+\bmu_0\bmu_0\trs )\\ -\low(
\nu_0\bSigma_0+\bmu_0\bmu_0\trs )\\
\kappa_0\bmu_0\end{bmatrix}
\end{equation}
and
\begin{equation}
H(\btheta) = (2\pi)^{-d/2},\quad
A(\bchi, \btau)=
-\frac{d}{2}\ln \kappa_0
-\frac{\nu_0+d+1}{2}\ln \frac{|\nu_0\bSigma_0|}{2^d}
+\ln \Gamma_d\left(\frac{\nu_0+d+1}{2}\right).
\end{equation}

\begin{cor}\label{res:quadsep}
Let $\bx_1,\ldots,\bx_n\in\R^d$ be pairwise distinct and let $\hat{\cI}$ be the MAP
partition of $\bx_1,\ldots,\bx_n$ in the Normal Bayesian Mixture Model where the
prior on component parameters is given by \eqref{eq:NIWmodel}.
Then the family $\{\bfx_I\colon I\in\hat{\cI}\}$ is \textit{quadratically}
separable, i.e. every two clusters are separated by a quadratic surface.
\end{cor}

\begin{proof}
It follows from \Cref{res:Tlinsep} and the formula for the sufficient statistic
$\bT$ in this model, \eqref{eq:alladyn}.
\end{proof}

\subsubsection{Normal (fixed covariance)}
Here the component covariance matrix is assumed to be known a priori; the
component mean is unknown and this is the parameter on which the prior
distribution is set, i.e. $\btheta=\bmu$, $\Theta=\R^d$ and
$\bx\cond\bmu\sim\cN(\bmu,\bSigma_0)$. The prior is
\begin{equation}\label{eq:Nmodel}
\begin{array}{rcl}
\bmu&\sim&\Normal(\bmu_0,\bPsi_0)
\end{array}
\end{equation}
The hyperparameters are $\bmu_0\in\R^d$ and
$\bPsi_0,\bSigma_0\in\cS^+$. This prior is also listed in \cite{bib:Gelman2013bayesian}. Clearly
\begin{equation}\label{eq:NN_expect}
\E(\bV(\bx\cond \bmu))=\bSigma_0,\quad
\bV(\E(\bx\cond \bmu))=\bV(\bmu)=\bPsi_0.
\end{equation}

The conditional densities are given by
\begin{equation}\label{eq:Nmodel_dens}
\begin{array}{rcl}
\bx\cond\bmu&\sim&
(2\pi)^{-d/2}|\bSigma_0|^{-1/2}
\exp\left\{-\re{2}(\bx-\bmu)\trs\bSigma_0^{-1}(\bx-\bmu)\right\}\\
\bmu&\sim&
(2\pi)^{-d/2}|\bPsi_0|^{-1/2}
\exp\left\{-\re{2}(\bmu-\bmu_0)\trs\bPsi_0^{-1}(\bmu-\bmu_0)\right\}
\end{array}
\end{equation}

The density of $\bx\cond\btheta$ can be expressed as \eqref{eq:compdens} by
placing
\begin{equation}\label{eq:mufasa}
\begin{split}
h(\bx)= ( 2\pi )^{-d/2}|\bSigma_0|^{-1/2}\exp\left\{-\re{2}\bx\trs\bSigma_0^{-1}\bx\right\}, \quad
\bT(\bx)&=\bSigma_0^{-1}\bx, \quad
\bta(\btheta) = \bmu,\\
B(\btheta)&=\re{2}\bmu\trs\bSigma_0^{-1}\bmu
\end{split}
\end{equation}

We get $\eqref{eq:expdef}$ by placing $\bma=\begin{bmatrix}\diag(\bSigma_0^{-1})\\\low(\bSigma_0^{-1})\end{bmatrix}$
\begin{equation}
\bmB(\btheta)=
\begin{bmatrix}\re{2}\diag(\bmu\bmu\trs)\\\low(\bmu\bmu\trs)\end{bmatrix},\quad
\btau= 
\begin{bmatrix}\diag(\bPsi_0^{-1})\\\low(\bPsi_0^{-1})\end{bmatrix},\quad
\bchi=\bPsi_0^{-1}\bmu_0
\end{equation}
and
\begin{equation}
H(\btheta) = (2\pi)^{-d/2},\quad
A(\bchi, \btau)= \re{2}\ln |\bPsi_0|+\re{2}\bmu_0\trs\bPsi_0\bmu_0.
\end{equation}

\begin{cor}\label{res:linsep}
Let $\bx_1,\ldots,\bx_n\in\R^d$ be pairwise distinct and let $\hat{\cI}$ be the MAP
partition of $\bx_1,\ldots,\bx_n$ in the Normal Bayesian Mixture Model where the
prior on component parameters is given by \eqref{eq:NIWmodel}.
Then the family $\{\bfx_I\colon I\in\hat{\cI}\}$ is \textit{linearly}
separable, i.e. every two clusters are separated by a hyperplane.
\end{cor}

\begin{proof}
It follows from \Cref{res:Tlinsep} and the formula for the sufficient statistic
$\bT$ in this model, \eqref{eq:alladyn} and the fact that $\bSigma_0$ is an
invertible matrix.
\end{proof}

\subsubsection{Normal-inverse-Gamma}
In Normal-inverse-Gamma model we assume that the base covariance matrix and the
component covariance matrix are known up to some scaling factor
$\lambda\sim\cG^{-1}(\beta_0+1,\beta_0)$. Hence the parameter is
$\btheta=(\bmu,\lambda)$, the parameter space is
$\Theta=\R^d\times\R\equiv\R^{d+1}$ and
\begin{equation}\label{eq:NIGmodel}
\begin{array}{rcl}
\lambda&\sim&\cG^{-1}(\beta_0+1,\beta_0)\\
\bmu\cond\lambda&\sim&\Normal(\bmu_0,\lambda \bPsi_0)\\
x\cond\bmu,\lambda&\sim&\Normal(\bmu,\lambda\bSigma_0)
\end{array}
\end{equation}
Here the hyperparameters are $\beta_0>0$, $\bmu_0\in\R^d$ and
$\bPsi_0,\bSigma_0\in\cS^+$. With this prior
\begin{equation}
\begin{array}{rl}
\E(\bV(\bx\cond \bmu,\lambda))&=\E\lambda\bSigma_0=\bSigma_0,\\
\bV(\E(\bx\cond \bmu,\lambda))&=\bV(\bmu)=\E\bV(\bmu\cond \lambda)+\bV\E(\bmu\cond\lambda)=
\E \lambda\bPsi_0+\bV(\bmu_0)=\bPsi_0.
\end{array}
\end{equation}

The conditional densities are given by
\begin{equation}\label{eq:NIGmodel_dens}
\begin{array}{rcl}
\bx\cond\bmu,\lambda&\sim&
(2\pi)^{-d/2}|\bSigma_0|^{-1/2}\lambda^{-d/2}
\exp\left\{-\re{2\lambda}(\bx-\bmu)\trs\bSigma_0^{-1}(\bx-\bmu)\right\}\\
\bmu\cond\lambda&\sim&
(2\pi)^{-d/2}|\bPsi_0|^{-1/2}\lambda^{-d/2}
\exp\left\{-\re{2\lambda}(\bmu-\bmu_0)\trs\bPsi_0^{-1}(\bmu-\bmu_0)\right\}\\
\lambda&\sim&
\beta_0^{\beta_0+1}\Gamma(\beta_0)^{-1}\lambda^{-( \beta_0+2 )}
\exp\left\{-\beta_0/\lambda\right\}
\end{array}
\end{equation}

The density of $\bx\cond\btheta$ can be expressed as \eqref{eq:compdens} by
placing
\begin{equation}\label{eq:skaza}
\begin{split}
h(x)\equiv ( 2\pi )^{-d/2}|\bSigma_0|^{-1/2}, \quad
\bT(\bx)&=\begin{bmatrix}-\re{2}\bx\trs\bSigma_0^{-1}\bx\\\bSigma_0^{-1}\bx\end{bmatrix}, \quad
\bta(\btheta) = \begin{bmatrix}1/\lambda\\\bmu/\lambda\end{bmatrix},\\
B(\btheta)&=\frac{d}{2}\ln \lambda+\re{2}\bmu\trs\bSigma_0^{-1}\bmu/\lambda
\end{split}
\end{equation}

We get $\eqref{eq:expdef}$ by placing $\bma=\begin{bmatrix}d/2\\\diag(\bSigma_0^{-1})\\\low(\bSigma_0^{-1})\end{bmatrix}$
\begin{equation}
\bmB(\btheta)=\begin{bmatrix}\ln\lambda\\\re{2}\diag(\bmu\bmu\trs)/\lambda\\\low(\bmu\bmu\trs)/\lambda\end{bmatrix},\quad
\btau= 
\begin{bmatrix}\beta_0+2\\\diag(\bPsi_0^{-1})\\\low(\bPsi_0^{-1})\end{bmatrix}
,\quad
\bchi=\begin{bmatrix}-\beta_0\\\bPsi_0^{-1}\bmu_0\end{bmatrix}
\end{equation}
and
\begin{equation}
H(\btheta) = (2\pi\lambda)^{-d/2},\quad
A(\bchi, \btau)= \re{2}\ln |\bPsi_0|+\re{2}\bmu_0\trs\bPsi_0\bmu_0
\end{equation}

\begin{cor}\label{res:elipsep}
Let $\bx_1,\ldots,\bx_n\in\R^d$ be pairwise distinct and let $\hat{\cI}$ be the MAP
partition of $\bx_1,\ldots,\bx_n$ in the Normal Bayesian Mixture Model where the
prior on component parameters is given by \eqref{eq:NIWmodel}.
Then the family $\{\bfx_I\colon I\in\hat{\cI}\}$ is \textit{elliptically}
separable, i.e. every two clusters are separated by a multidimensional ellipse.
\end{cor}

\begin{proof}
It follows from \Cref{res:Tlinsep} and the formula for the sufficient statistic
$\bT$ in this model, \eqref{eq:alladyn}.
\end{proof}
\subsubsection{Comparison of the models}
If we assume that the component parameters in the Normal Bayesian Mixture Model
are distributed by \eqref{eq:Nmodel} then we assume that the covariance matrix
in each component is equal to $\bSigma_0$ which is known to us. The results of
\cite{bib:Rajkowski2018} imply that the misspecification of this hyperparameter may
lead to serious inference issues regarding the number of clusters, at least as far as the MAP partition is
concerned. In the light of these findings, \eqref{eq:NIWmodel} seems to be a
safer choice of the prior for the component parameters. In this case the
covariance matrix is chosen independently for each component according to the
inverse-Wishart distribution. Note that although the Normal-inverse-Wishart
prior gives more flexibility in terms of the component covariances, it imposes
some modelling restriction, namely the expected within and between group covariance
matrices are proportional, as is shown by \eqref{eq:NIW_expect}. This does not
affect the fixed-covariance model, cf. \eqref{eq:NN_expect}.

\smallskip
This is the reason for which we propose the Normal-inverse-Gamma prior. It is
not listed in \cite{bib:Gelman2013bayesian} and we were not able to find any
reference to it in the literature. It only allows a 1-parameter variation of the covariance function, but
no restrictions are imposed on the within-group means, unlike the Normal-inverse-Wishart prior.
At the same time, by allowing the component covariance matrix to scale between
clusters can be a remedy to the drawbacks of fixed covariance prior that were
pointed out in \cite{bib:Rajkowski2018}.

\smallskip
As a final point we note that Normal-inverse-Gamma prior is a generalisation of
the Normal prior in the sense that \eqref{eq:NIGmodel} becomes \eqref{eq:Nmodel}
as $\beta_0\to\infty$. Analogously, Normal-inverse-Wishart prior is a quasi-extension
of the Normal prior, since as $\eta_0\to\infty$, \eqref{eq:NIWmodel} converges
to \eqref{eq:Nmodel}, but with $\bPsi_0=\bSigma_0/\kappa_0$.

\section{Discussion of potential applications}
We proved a separability result concerning the MAP partition in the Exponential
Family Bayesian Mixture Models. In particular, we proved
linear or quadratic separability of the MAP partition in most popular
Normal Bayesian Mixture Models. Apart from an aesthetic analogy to the properties of
Fisher Discriminant Analysis, the benefits of such result may be twofold.

\smallskip
In \cite{bib:Rajkowski2018} the linear separability of the MAP partition is
crucial for establishing the existence of `limits' of the MAP partitions when
the prior on partitions is the Chinese Restaurant Process and the data is
independently and identically distributed with some `input distribution'. 
The limit is related to the partitions of observation space which maximises a given 
functional $\Delta$ (which depends only on the hypeerparameter $\bSigma_0$ and
the input distribution). The linear separability is important for two reasons:
firstly, it is possible to consider the limits of sequences of convex sets and
secondly: it is possible to apply the Uniform Law of Large Numbers for the
family of convex sets. \Cref{res:Tlinsep} should enable an analogous
approach for the Normal-inverse-Wishart and Normal-inverse-Gamma priors on the
component parameters.

\smallskip
The other kind of application is more practical; \Cref{res:Tlinsep} shows that
the search for an MAP partition may be restricted to situations where clusters are
quadratically separated. The space of such partitions is still
far too large for an exhaustive search, but may help in finding a partition
whose score approximates the MAP score.

\section{Proofs}\label{sec:proofs}
\subsection{Proof of \Cref{res:Tlinsep}}
\begin{lem}
\label{res:localMax}
Let $\cL$ be a family of real functions on $\R^d$. Let $\bx_1,\ldots,\bx_n\in \R^d$
and let $\hat{I}$ be the MAP partition for $\bx_1,\ldots,\bx_n$ in some Bayesian
Mixture Model, given by \eqref{eq:bmm3}. If for any $\LL\subset [n]$, $k,l\in\N$ such that $k+l= |\LL|$ and
$\mnt{I}_{\Ik,\LL}$ such that
\begin{equation}\label{eq:ikudef}
\mnt{I}_{\Ik,\LL}\in\argmax_{I\subset \LL: |I|=\Ik} \big(\ln f_{\Ik}(\bfx_I)+ \ln
f_l(\bfx_{\LL\sm I})\big)
\end{equation}
observations $\bfx_{\mnt{I}_{\Ik,\LL}}$ and $\bfx_{\LL\sm\mnt{I}_{\Ik,\LL}}$ are separated by
$\cL$ then $\{\bfx_I\colon
I\in\hat{\cI}\}$ is separated by $\cL$.
\end{lem}
\begin{proof}
Firstly note that by \eqref{eq:MAPdef_argmax}
\begin{equation}
\hat{\cI}\in\argmax_{\textrm{partitions $\cI$ of $[n]$}} 
\Big(\ln p_n(\cI)+\sum_{I\in\cI} \ln f_{|I|}(\bfx_I)\Big).
\end{equation}
Assume that the assumptions of \Cref{res:localMax} hold. Suppose that $\hat{\cI}$ is not
separated by $\cL$. Then there exist $\mnt{I},\mnt{J}\in \hat{\cI}$ such that
$\bfx_\mnt{I}$ and $\bfx_\mnt{J}$ are not separated by $\cL$. Let
$\LL=\mnt{I}\cup \mnt{J}$ and $\Ik=|\mnt{I}|$. Let $\tilde{I}=I_{\Ik,\LL}$ and
$\tilde{J}=\LL\sm \tilde{I}$. Moreover let $\tilde{\cI}$ be a partition of
$[n]$ obtained by replacing $\mnt{I},\mnt{J}$ by $\tilde{I},\tilde{J}$, i.e.
$\tilde{\cI}=\hat{\cI}\sm\{\mnt{I},\mnt{J}\}\cup \{\tilde{I},\tilde{J}\}$. Note
that $p_n(\hat{\cI})=p_n(\tilde{\cI})$ (we have $|\hat{I}|=|\tilde{I}|$ and
$|\hat{J}|=|\tilde{J}|$, so we use the exchangeability of $\bm{\Pi}$). Moreover
$\bfx_{\mnt{I}}$ and $\bfx_{\mnt{J}}$ are not separated by $\cL$ so by the
assumptions of \Cref{res:localMax}
\begin{equation}
\mnt{I}\notin
\argmax_{I\subset \LL: |I|=\Ik} \big(\ln f_k(\bfx_I)+ \ln
f_l(\bfx_{\LL\sm
I})\big)
\end{equation}
and hence
\begin{equation}
\ln f_k(\bfx_{\tilde{I}}) + \ln f_l(\bfx_{\tilde{J}})>
\ln f_k(\bfx_{\mnt{I}}) + \ln f_l(\bfx_{\mnt{J}}).
\end{equation}
This means that
\begin{equation}
\ln p_n(\tilde{\cI})+\sum_{I\in\tilde{\cI}} \ln f_{|I|}(\bfx_I)>
\ln p_n(\hat{\cI})+\sum_{I\in\hat{\cI}} \ln f_{|I|}(\bfx_I),
\end{equation}
which contradicts the definition of $\hat{\cI}$ and the proof follows.
\end{proof}
\begin{lem}
\label{res:convexMaxSum}
Let $\convset\subseteq \R^D$ be a convex set. Let $\alpha\colon \convset\to \R$ be a
strictly concave function, $z_1,\ldots,z_{k+l}\in\R^D$ are pairwise distinct. 
If $\sum_{i\in I} z_i \in\convset$ for every $I\subseteq [k+l]$ such that
$|I|=k$ and
\begin{equation}
\mnt{J}\in\argmin_{I\subset [k+l]\colon |I|=\Ik} \alpha\big(\sum_{i\in I} z_i\big) 
\end{equation}
then $\bz_{\mnt{J}}$ and $\bz_{[k+l]\sm\mnt{J}}$ are linearly separable.
\end{lem}
\begin{proof}
Consider the set of all possible sums of $\Ik$ distinct vectors $z_i$, i.e. $\cS_\Ik=\{\sum_{i\in I} z_i\colon I\subset [n],|I|=\Ik\}$ and
let $\hat{s}_\Ik\in\argmin_{s\in\cS_\Ik} \alpha(s)$. Since $\alpha$ is strictly concave, then
$\hat{s}_\Ik$ is a vertex of $\conv\ {\cS_\Ik}$.
This means that there exist a vector $v_0\in\R^d$ such that
$\hat{s}_\Ik\in\argmax_{s\in \cS_\Ik} \langle s, v_0\rangle$ (cf. \cite{bib:Moszynska2006}, Corollary 3.3.6), where $\langle\cdot,\cdot\rangle$ is a standard Euclidean
scalar product. I can also choose $v_0$ so that $\langle z_i, v_0\rangle$ are
are all different (because we are dealing with a discrete set). Let
$z_{(1)},\ldots,z_{(k+l)}$ be a decreasing ordering of vectors
$z_i$ `in the direction $v_0$', i.e. $\{
z_{(1)},\ldots,z_{(k+l)}\}=\{z_1,\ldots,z_{k+l}\}$ and
$\langle z_{(i)},v_0\rangle >\langle z_{(j)},v_0\rangle$ if $i<j$. Note that
\begin{equation}
\Big\langle\sum_{i\in I} z_i,v_0\Big\rangle=
\sum_{i\in I} \langle z_i,v_0\rangle
\end{equation}
and therefore $\mnt{I}=\{z_{(1)},\ldots,z_{(k)}\}$. Thus the sets
$\{z_i\colon i\in \mnt{I}\}$ and $\{z_i\colon i\notin \mnt{I}\}$ are linearly
separated by the hyperplane $\{u\in\R^D\colon \langle u,v_0\rangle=\langle
z_{(k)}+z_{(k+1)},v_0\rangle/2\}$.

\begin{rem}\label{res:Aconv}
Let us assume the notation of \Cref{sec:exp_models}. Then the set $\Omega$ is a
convex set and the function $A\colon\Omega\to\R$ is strictly convex.
\end{rem}
\begin{proof}
It is a well known property of exponential families in canonical form
(cf. \cite{bib:diaconis1979conjugate}).
\end{proof}

\textit{Proof of \Cref{res:Tlinsep}.} 
Let $\LL\subseteq [n]$, $\Ik,l\in\N$ and $\mnt{I}_{\Ik,\LL}$ be as in
\Cref{res:localMax}. Plugging the formula
\eqref{eq:expmarg} into \eqref{eq:ikudef} gives:
\begin{equation}\label{eq:margolcia}
\begin{split}
\mnt{I}_{\Ik,\LL} &=
\argmax_{I\subset \LL\colon |I|=\Ik} 
\Big(
\sum_{i\in I} \ln h(\bx_i) + A(\bchi_{\bfx_I}, \btau_k) -
A(\bchi,\btau)+\\
&\hspace*{1.5cm}+
\sum_{i\in \LL\sm I} \ln h(\bx_i) + A(\bchi_{\bfx_{\LL\sm I}}, \btau_l) -
A(\bchi,\btau)
\Big)=\\
&=
\argmax_{I\subset \LL\colon |I|=\Ik} 
\Big(
\sum_{i\in \LL} \ln h(\bx_i) + A(\bchi_{\bfx_I}, \btau_k)
+ A(\bchi_{\bfx_{\LL\sm I}}, \btau_l) - 2A(\bchi,\btau)
\Big)=\\
&=
\argmax_{I\subset \LL\colon |I|=\Ik} 
\Big(
A(\bchi_{\bfx_I}, \btau_k)
+ A(\bchi_{\bfx_{\LL\sm I}}, \btau_l) 
\Big)
\end{split}
\end{equation}
Let
$\bt_i=\begin{bmatrix}\bT(\bx_i),\bma\end{bmatrix}$ and let
$\bt_0=\begin{bmatrix}\bchi,\btau\end{bmatrix}$ and
$\bt_\cU=\sum_{i\in\cU} \bt_i$.
By \Cref{res:Aconv}, $A$ is a strictly convex function. Hence the functions
$\un{\alpha}(\bt)=A(\bt_0+\bt)$ and $\ov{\alpha}(\bt)=A(\bt_0+\bt_\cU-\bt)$ are
also strictly convex and so
is their sum, $\alpha(\bt)=\un{\alpha}(\bt)+\ov{\alpha}(\bt)$. By
\eqref{eq:margolcia} we get
\begin{equation}
\mnt{I}_{\Ik,\LL} =
\argmax_{I\subset \LL\colon |I|=\Ik}
\alpha\Big(\sum_{i\in I} \bt_i\Big)
\end{equation}
Therefore by \Cref{res:convexMaxSum} we obtain that
$\bt_{\mnt{I}_{\Ik,\LL}}$ and
$\bft_{\LL\sm\mnt{I}_{\Ik,\LL}}$ are linearly separable (i.e. in
terms of the base functions). Obviously this yields $\bT$-linear separability of 
$\bfx_{\mnt{I}_{\Ik,\LL}}$ and
$\bfx_{\LL\sm\mnt{I}_{\Ik,\LL}}$ and the proof follows.\qed
\end{proof}
\bibliographystyle{named}
\bibliography{bib}
\end{document}